\documentclass[a4paper,reqno,12pt]{amsart}
\usepackage{verbatim}
\usepackage{amssymb,amsmath,amsthm}
\usepackage{amsfonts}
\usepackage{array}
\usepackage{xcolor}
\usepackage{youngtab}

\newcommand{\qbin}[2]{\genfrac{[}{]}{0pt}{}{#1}{#2}_{q}}
\newcommand{\sfk}[2]{\genfrac{[}{]}{0pt}{}{#1}{#2}}

\newtheorem{theorem}{Theorem}[section]
\newtheorem{lemma}{Lemma}[section]
\newtheorem{corollary}{Corollary}[section]
\newtheorem{conjecture}{Conjecture}[section]
\newtheorem{proposition}{Proposition}[section]
\newtheorem{example}{Example}[section]

\newtheoremstyle{remark}
    {\dimexpr\topsep/2\relax} % space above
    {\dimexpr\topsep/2\relax} % space below
    {}          % body font
    {}          % indent amount
    {\bfseries} % theorem head font
    {.}         % punctuation after theorem head
    {.5em}      % space after theorem head
    {}          % theorem hed spec. (empty = "normal")

\theoremstyle{remark}
\newtheorem{remark}{Remark}[section]

\makeatletter
\@namedef{subjclassname@2020}{%
  \textup{2020} Mathematics Subject Classification}
\makeatother

\begin{document}

\title[]
{Extensions of MacMahon's sums of divisors}

\author{Tewodros Amdeberhan} 

\address{Department of Mathematics,
Tulane University, New Orleans, LA 70118, USA}
\email{tamdeber@tulane.edu}

\author{George E. Andrews}\thanks{The second author is partially supported by Simon Foundation Grant 633284.}

\address{Department of Mathematics, Penn State University, University Park, PA 16802, USA}
 \email{gea1@psu.edu} 

\author{Roberto Tauraso}

\address{Dipartimento di Matematica, % \\
Università di Roma ``Tor Vergata'', 00133 Roma, Italy}
\email{tauraso@mat.uniroma2.it}



\begin{abstract} 
In 1920, P. A. MacMahon generalized the (classical) notion of divisor sums by relating it to the theory of partitions of integers. In this paper, we extend the idea of MacMahon. In doing so we reveal a wealth of divisibility theorems and unexpected combinatorial identities. Our initial approach is quite different from MacMahon and involves \emph{rational} function approximation to 
MacMahon-type generating functions. One such example involves multiple $q$-harmonic sums
\begin{equation*}
\sum_{k=1}^n\frac{(-1)^{k-1}\qbin{n}{k}(1+q^k)q^{\binom{k}2+tk}}{[k]_q^{2t}\qbin{n+k}{k}}
=\sum_{1\leq k_1\leq\cdots\leq k_{2t}\leq n}\frac{q^{n+k_1+k_3\cdots+k_{2t-1}}+q^{k_2+k_4+\cdots+k_{2t}}}{[n+k_1]_q[k_2]_q\cdots[k_{2t}]_q}.
\end{equation*}
\end{abstract} 

\maketitle

\section{Introduction}

\noindent
In his 1920 paper, \emph{Divisors of Numbers and their Continuations in the Theory of Partitions}, P. A. MacMahon \cite{MacMahon} links the theory of integer partitions to divisor sums as follows. He begins with the simple observation that the number of partitions of $n$ in which all parts are identical \emph{equals} to the number of divisors of $n$. From here, it is natural to look at partitions of $n$ in which there are exactly two different sizes of parts; then on to three sizes, etc.. 

\smallskip
\noindent
Indeed, if we want the generating functions for $t$ sizes, it is clearly
\begin{align} \label{Mac1}
\sum_{1\leq k_1<k_2<\cdots<k_t}\frac{q^{k_1+k_2+\cdots+k_t}}{(1-q^{k_1})(1-q^{k_2})\cdots(1-q^{k_t})}.
\end{align}

\noindent
Immediately, MacMahon's knowledge of elliptic functions comes into play. He recognized that now generalizing to sum of divisors along these lines should be of  substantial interest because of the connection to elliptic functions. Indeed, MacMahon proved that
\begin{align} \label{Mac2} \nonumber
U_t(q)=\sum_{n\geq0}MO(t,n)q^n:&=\sum_{1\leq k_1<k_2<\cdots<k_t}\frac{q^{k_1+k_2+\cdots+k_t}}{(1-q^{k_1})^2(1-q^{k_2})^2\cdots(1-q^{k_t})^2} \\ 
&=\frac{(-1)^t}{2^{2t}(2t+1)!}\frac1{\mathbf{J}_1}\,\mathbf{J}(\mathbf{J}^2-1^2)(\mathbf{J}^2-3^2)\cdots(\mathbf{J}^2-(2t-1)^2).
\end{align}
The above identity \eqref{Mac2} is to be interpreted umbrally, i.e. after the multiplication is carried out on the right-hand side each $\mathbf{J}^s$ is to be replaced by $\mathbf{J}_s$, and 
\begin{align} \label{Mac3}
\mathbf{J}_s=\sum_{m\geq0}(-1)^m(2m+1)^sq^{\binom{m+1}2}.
\end{align}
We observe that, while the evolution of MacMahon's project is still natural, one might also consider 
\begin{align} \label{AAT1}
V_t(q)=\sum_{n\geq0}M(t,n)q^n:&=\sum_{1\leq k_1\leq k_2\leq\cdots\leq k_t}\frac{q^{k_1+k_2+\cdots+k_t}}{(1-q^{k_1})^2(1-q^{k_2})^2\cdots(1-q^{k_t})^2}.
\end{align}
We note that when $t=1$ this again is the generating function $\sum_{n\geq 0}\sigma_1(n)q^n$ for the sum of the divisors of $n$.

\smallskip
\noindent
Now MacMahon treats extensively the relationship of $U_t(q)$ and analogous functions to elliptic functions, and one of his conjectures is settled in
\cite{Andrews-Rose}. However, in our paper here we reveal a rich arithmetic aspect of these generalized divisor functions that has been overlooked. For example, look at the the next congruences in arithmetic progression
\begin{align} \label{M2andMO2}
M(2,5n+1)\equiv MO(2,5n+1)\equiv 0 \pmod5.
\end{align}
Indeed many more such congruences occur for moduli $5$ and $7$ (cf. Theorem \ref{cong1}). 

\smallskip
\noindent
In addition to surprising congruences, we also obtained a novel identity among these generalized divisor sums. Namely (cf. Corollary \ref{tranMac}):
\begin{align} \label{partofcor3.1}
\sum_{1\leq k_1\leq k_2\leq\cdots\leq k_t}\frac{q^{k_1+k_2+\cdots+k_t}}{\prod_{j=1}^t(1-q^{k_j})^2}
&=\sum_{1\leq k_1\leq k_2\leq\cdots\leq k_{2t-1}}\frac{k_1\, q^{k_1+k_3+\cdots+k_{2t-1}}}{\prod_{j=1}^{2t-1}(1-q^{k_j})}.
\end{align}
The left-hand side is the generating function for $M(t,n)$ by definition. As an example, show that $M(2,4)=14$. We examine partitions with exactly two sizes of parts (possibly equal). Thus the relevant partitions are $3^11^1, 2^12^1, 2^11^2, 1^31^1, 1^21^2,1^11^3$ (the exponents denote frequency of appearance), and we get the sum of the product of the frequencies $1\cdot1+1\cdot1+1\cdot2+3\cdot1+2\cdot2+1\cdot3=14$.

\smallskip
\noindent
On the other hand, for $2t-1=3$, the right-hand is considering partitions with three part sizes where the second part need not appear at all, namely 
$$3^13^01^1, 3^11^01^1, 2^12^02^1, 2^11^11^1, 2^11^01^2, 2^12^01^2, 1^31^01^1, 1^21^01^2, 1^21^11^1, 1^11^21^1, 1^11^11^2, 1^11^01^3$$ 
and the sum of the smallest parts is $1+1+1+2+1+1+1+1+1+1+1+1+1=14$. We shall discuss the possibilities suggested by this result in the conclusion of the paper.

\smallskip
\noindent
Finally we should say a word about the methods we use. Our work was inspired by identities discovered by Karl Dilcher \cite{D}. In the next section we outline the path from Dilcher's results to Corollary \ref{tranMac}. The following two sections are devoted to the proof of Theorem \ref{T3} from which 
Corollary \ref{tranMac} is derived. Section 5 garners another variant to those from the earlier sections. The next two sections lead to the wonderful congruences, of Section $8$, for both $M(t,n)$ and $MO(t,n)$. We conclude with a section devoted to open problems and comments on salient points revolving identity \eqref{partofcor3.1}.

\section{Paving the road to Theorem 3.1} 

\noindent 
In \cite[Theorem 4]{D}, K. Dilcher established an interesting $q$-identity: for arbitrary positive integers $n$ and $t$, there holds
\begin{equation}\label{Did}
\sum_{k=1}^n \frac{(-1)^{k-1}\qbin{n}{k}q^{\binom{k}{2}+tk}}{[k]_q^{t}}
=\sum_{1\leq k_1\leq \cdots \leq k_{t}\leq n}\frac{q^{k_1+k_2+\dots+k_{t}}}{[k_1]_q [k_2]_q \cdots [k_{t}]_q}
\end{equation}
where $[x]_q :=\frac{1-q^x}{1-q}$ and $\qbin{x}{k}=\frac{[x]_q[x-1]_q\cdots [x-k+1]_q}{[1]_q[2]_q\cdots [k]_q}$ are the Gaussian $q$-binomial coefficients.
On the other hand, according to \cite[Theorem 2.3]{MSS},
$$\sum_{k=1}^n \frac{(-1)^{k-1}\qbin{n}{k}[k]_qq^{\binom{k}{2}-k(n-1)}}{\qbin{x+k}{k}}
\!\!\!\!\!
\sum_{1\leq k_1\leq \cdots \leq k_{t}\leq n}\frac{q^{k_1+k_2+\dots+k_{t}}}{[x+k_1]_q [x+k_2]_q \cdots [x+k_{t}]_q}
=\frac{q^{tn}[n]_q}{[x+n]_q^{t+1}}$$
which is equivalent to, see also \cite[Corollary 3.3]{X} and replace $x$ by $q^x$ there,
\begin{equation}\label{MSSid}
\sum_{k=1}^n \frac{(-1)^{k-1}\qbin{n}{k}[k]_qq^{\binom{k}{2}+tk}}{[x+k]_q^{t+1}}
=\frac{1}{\qbin{x+n}{n}}\sum_{1\leq k_1\leq \cdots \leq k_{t}\leq n}\frac{q^{k_1+k_2+\dots+k_{t}}}{[x+k_1]_q [x+k_2]_q \cdots [x+k_{t}]_q}
\end{equation}
via the $q$-inverse pair formula
$$b_n=\sum_{k=1}^n(-1)^{k-1}\qbin{n}{k}a_k\quad\Leftrightarrow\quad
a_n=\sum_{k=1}^n(-1)^{k-1} q^{\binom{n-k}{2}}\qbin{n}{k}b_k.$$
Notice that the identity \eqref{MSSid} is a generalization of \eqref{Did} by setting $x=0$.

\noindent
We show the rational multiple sum identity
\begin{equation}\label{mhsid}
\sum_{k=1}^n \frac{(-1)^{k-1}\binom{n}{k}}{k^t\binom{x+k}{k}}
=\sum_{1\leq k_1\leq \cdots \leq k_{t}\leq n} \frac{k_1}{x+k_1}\cdot
\frac{1}{k_1k_2 \cdots k_t},
\end{equation}
and then we find two $q$-analogs of \eqref{mhsid} which are some variations of the identity \eqref{MSSid}. 

\noindent 
The first $q$-identity is
\begin{equation}\label{ATidA}
\sum_{k=1}^n\frac{(-1)^{k-1}\qbin{n}{k}(1+q^k)q^{\binom{k}2+tk}}{[k]_q^{2t}\qbin{n+k}{k}}
=\sum_{1\leq k_1\leq k_2\leq\cdots\leq k_{2t}\leq n}\frac{q^{n+k_1+k_3\cdots+k_{2t-1}}+q^{k_2+k_4+\cdots+k_{2t}}}
{[n+k_1]_q[k_2]_q\cdots[k_{2t}]_q},
\end{equation}
and the second one is
\begin{equation}\label{ATidB}
\sum_{k=1}^n \frac{(-1)^{k-1}\qbin{n}{k}q^{\binom{k}{2}+(x+2t)k}}{[k]_q^{2t}\qbin{x+k}{k}}
=\sum_{1\leq k_1\leq k_2\leq \cdots \leq k_{2t}\leq n}\frac{q^{(x+k_1)+k_2+\dots+k_{2t}}}{[x+k_1]_q [k_2]_q \cdots [k_{2t}]_q}.
\end{equation}
We also provide a new proof of the identity \cite[equation (23)]{HP}:
\begin{equation*}%\label{HPid}
\sum_{1\leq k_1\leq k_2\leq\cdots\leq k_t\leq n}\frac{q^{k_1+k_2+\cdots+k_t}}{[k_1]_q^2[k_2]_q^2\cdots[k_t]_q^2}
=\sum_{k=1}^n\frac{(-1)^{k-1}\qbin{n}{k}(1+q^k)q^{\binom{k}2+tk}}{[k]_q^{2t}\qbin{n+k}{k}}.
\end{equation*}

\section{Generalizing the rational multiple sum identity \eqref{mhsid}}

\noindent
We start by proving a very useful preliminary result.

\begin{lemma} \label{Master} If two sequences are related by $\sum_{k=1}^n(-1)^{k-1}\binom{n}k\,a_k=b_n$ then, for all positive integers $n$ and $t$, 
$$\sum_{k=1}^n\frac{(-1)^{k-1}\binom{n}k\,a_k}{(z+k)^{t}}
=\frac{1}{\binom{z+n}{n}}\sum_{1\leq k_1\leq\cdots\leq k_{t}\leq n}\frac{b_{k_1}\,\binom{z+k_1}{k_1}}{\prod_{j=1}^{t}(z+k_j)}.$$
\end{lemma}
\begin{proof} We may readily represent the given hypothesis in an infinite matrix form
\begin{align}\label{hypo} \left[(-1)^{j-1}\binom{i}j\right]_{i,j=1}^{\infty}\mathbf{a}^{T}=\mathbf{b}^T \end{align}
where $\mathbf{a}=[a_1,a_2,\dots]$ and $\mathbf{b}=[b_1,b_2,\dots]$ are row matrices. Introduce two more matrices defined by
\begin{align*} 
\left[\frac{\binom{z+j}{j}}{(z+j)\binom{z+i}{i}}\cdot \delta_{i\geq j}\right]_{i,j=1}^{\infty} \qquad \text{and} \qquad
\left[\frac{(-1)^{j-1}\binom{i}j}{z+j}\right]_{i,j=1}^{\infty}
\end{align*}
and we claim the matrix product identity that
\begin{align} \label{step1}
\left[\frac{\binom{z+k}{k}}{(z+k)\binom{z+i}{i}} \cdot \delta_{i\geq k}\right]_{i,k=1}^{\infty}\pmb{\cdot} \left[(-1)^{j-1}\binom{k}j\right]_{k,j=1}^{\infty}
=\left[\frac{(-1)^{j-1}\binom{i}j}{z+j}\right]_{i,j=1}^{\infty}, \end{align}
which is equivalent to the binomial sum identity
$$\sum_{k=m}^n\frac{\binom{z+k}{k}}{z+k}\cdot\binom{k}m=\frac{\binom{z+n}{n}}{m+z}\cdot \binom{n}m, \qquad
\text{for $n, m\geq1$}.$$
If we denote the summand by $F(m,k):=\frac{\binom{z+k}{k}}{z+k}\cdot\binom{k}m$ and $G(m,k):=F(m,k)\cdot\frac{k-m}{z+m}$ then one can check that
$F(m,k)=G(m,k+1)-G(m,k)$ and hence (after some simplifications)
$$\sum_{k=m}^nF(m,k)=G(m,n+1)-G(m,m)=G(m,n+1)=\frac{\binom{z+n}{n}}{z+m}\cdot\binom{n}m,$$
as desired. The case $t=1$ of the Lemma follows from (\ref{hypo}) and (\ref{step1}) because we now have
\begin{align*}
\left[\frac{(-1)^{j-1}\binom{i}j}{z+j}\right]_1^{\infty} \pmb{a}^T=
\left[\frac{\binom{z+k}{k} \cdot \delta_{i\geq k}}{(z+k)\binom{z+i}{i}}\right]_1^{\infty}\left[(-1)^{j-1}\binom{k}j\right]_1^{\infty}\pmb{a}^T
=\left[\frac{\binom{z+k}{k} \cdot \delta_{i\geq k}}{(z+k)\binom{z+i}{i}}\right]_1^{\infty}\pmb{b}^T, \end{align*}
which tantamount 
$$\sum_{k=1}^n\frac{(-1)^{k-1}\binom{n}k\,a_k}{z+k}
=\frac{1}{\binom{z+n}{n}}\sum_{k=1}^n\frac{b_{k}\,\binom{z+k}{k}}{z+k}.$$

\noindent
The general case (of any $t$) is achieved by a repeated application of what we proved for $t=1$.  The proof is complete.
\end{proof}

\noindent
Next, we are able to derive the following corollary. Notice that \eqref{mhsid} would now become the special case for $z=0$.

\begin{corollary} For all positive integers $n$ and $t$, we have
$$\sum_{k=1}^n\frac{(-1)^{k-1}\binom{n}k}{(z+k)^{t}\binom{x+k}k}
=\frac{1}{\binom{z+n}{n}}\sum_{1\leq k_1\leq\cdots\leq k_{t}\leq n}\frac{k_1\binom{z+k_1}{k_1}}{(x+k_1)\prod_{j=1}^{t}(z+k_j)}.$$
\end{corollary}
\begin{proof} Choose $a_k=\frac1{\binom{x+k}k}$ and $b_k=\frac{k}{x+k}$ to verify the hypothesis $\sum_{k=1}^n(-1)^{k-1}\frac{\binom{n}k}{\binom{x+k}k}=\frac{n}{x+n}$ is satisfied. This, however, cane be proved via the Wilf-Zeilberger (WZ) pair: if we let $F(n,k)=(-1)^{k-1}\frac{\binom{n}k}{\binom{x+k}k}$ and $G(n,k)=\frac{x+k}{x+n}F(n,k)$ then $F(n,k)-G(n,k+1)-G(n,k)$ can easily be checked. Now, apply Lemma \ref{Master}.
\end{proof}

\smallskip
\begin{remark} Setting $x=n$, and by \cite[Theorem 2.3]{HPT}, or letting $q\to 1$ in Theorem \ref{T3} (shown below), we get
$$\sum_{1\leq k_1\leq \cdots \leq k_{t}\leq n} \frac{1}{k_1^2 k_2^2\cdots k_{t}^2}
=2\sum_{k=1}^n \frac{(-1)^{k-1}\binom{n}{k}}{k^{2t}\binom{n+k}{k}}
=\sum_{1\leq k_1\leq k_2\leq \cdots \leq k_{2t}\leq n}\frac{2}{(n+k_1)k_2 \cdots k_{2t}}.$$
Moreover, taking the limit as $n\to\infty$, we find
\begin{align*}
\sum_{\mathcal{A}_t}\frac{1}{k_1^2 k_2^2\cdots k_{t}^2}&=
2\lim_{n\to\infty}\sum_{1\leq k_1\leq k_2\leq \cdots \leq k_{2t}\leq n}
\frac{1}{(1+\frac{k_1}{n})\frac{k_2}{n} \cdots \frac{k_{2t}}{n}}\cdot \frac{1}{n^{2t}}\\
&=\int_{T_{2t-1}}\frac{\ln(1+x_1)}{x_1x_2\dots x_{2t-1}}\,dx_1dx_2\dots dx_{2t-1},
\end{align*}
where $\mathcal{A}_t$ (unbounded set of lattice points) and $\mathcal{T}_t$ (a simplex) are defined as
\begin{align*} \mathcal{A}_t&=\{(k_1,\dots,k_t)\in\mathbb{N}^t: k_1\leq k_2\leq \cdots\leq k_t\} \qquad \text{and} \\
T_{t}&=\{(x_1,x_2,\dots,x_{t}):0<x_1\leq x_2\leq \dots\leq x_{t}\leq 1\},
\end{align*}
respectively. On the other hand, 
\begin{align*}
\int_{T_{2t-1}}\frac{\ln(1+x_1)\,dx_1dx_2\cdots dx_{2t-1}}{x_1x_2\dots x_{2t-1}}
&=\sum_{k=1}^{\infty}\frac{(-1)^{k-1}}{k}\int_{T_{2t-1}}\frac{x_1^{k-1}dx_1dx_2\dots dx_{2t-1}}{x_2\dots x_{2t-1}}\\
&=2\sum_{k=1}\frac{(-1)^{k-1}}{k^{2t}}=(1-2^{1-2t})\zeta(2t).
\end{align*}
Hence we recover this known result (see for example  \cite[Remark 3]{HP})
$$\sum_{\mathcal{A}_t}\frac{1}{k_1^2 k_2^2\cdots k_{t}^2}=2(1-2^{1-2t})\zeta(2t)$$
expressed in terms of the \emph{Riemann zeta function} $\zeta(s)$.
\end{remark}

\bigskip
\noindent

\section{A $q$-analogue of \eqref{mhsid}}

\noindent
In the present section, we upgrade the rational identity \eqref{mhsid} into a $q$-analogue generalization. To this end, we define the three finite sums (suppressing the variable $q$):
\begin{align*}
F_t(n):&=\sum_{1\leq k_1\leq\cdots\leq k_t\leq n}\frac{q^{k_1+k_2+\cdots+k_t}}{[k_1]_q^2[k_2]_q^2\cdots[k_t]_q^2},\\
G_t(n):&=\sum_{k=1}^n\frac{(-1)^{k-1}(1+q^k)q^{\binom{k}2+tk}\qbin{n}{k}}{[k]_q^{2t}\qbin{n+k}{k}}
=\frac{1}{\qbin{2n}{n}}\,\,\sum_{k=1}^n\frac{(-1)^{k-1}(1+q^k)q^{\binom{k}2+tk}\qbin{2n}{n-k}}{[k]_q^{2t}}, \\
H_t(n):&=\sum_{1\leq k_1\leq\cdots\leq k_{2t}\leq n}\frac{q^{n+k_1+k_3\cdots+k_{2t-1}}+q^{k_2+k_4+\cdots+k_{2t}}}{[n+k_1]_q[k_2]_q\cdots[k_{2t}]_q}.
\end{align*}

\smallskip
\noindent
Now we are ready to state the promised generalization as a triplet identity
which implies the identity in \eqref{ATidA} and \cite[equation (23)]{HP}:
\begin{theorem} \label{T3}
For integers $n, t\geq 1$, we have $F_t(n)=G_t(n)=H_t(n)$.
\end{theorem}
\begin{proof} We start by proving $F_t(n)-F_t(n-1)=\frac{q^n}{[n]_q^2}F_{t-1}(n)$. Indeed, separate the case $k_t=n$ so that two sums emerge
\begin{align*} F_t(n)
&=\sum_{1\leq k_1\leq\cdots\leq k_{t-1}\leq n}\frac{q^{k_1+k_2+\cdots+k_{t-1}+n}}{[k_1]_q^2[k_2]_q^2\cdots[k_{t-1}]_q^2[n]_q^2}
+\sum_{1\leq k_1\leq\cdots\leq k_t\leq n-1}\frac{q^{k_1+k_2+\cdots+k_t}}{[k_1]_q^2[k_2]_q^2\cdots[k_t]_q^2}
\end{align*}
which translates to $F_t(n)=\frac{q^n}{[n]_q^2}F_{t-1}(n)+F_t(n-1)$, as desired.

\smallskip
\noindent
Let's verify that $G_t(n)$ satisfies the same recurrence as that of $F_t(n)$. First, observe that
\begin{align*} \frac{\qbin{n}{k}}{\qbin{n+k}{k}}-\frac{\qbin{n-1}{k}}{\qbin{n+k-1}{k}}
&=\frac{[n]_q!^2}{[n-k]_q![n+k]_q!}-\frac{[n-1]_q!^2}{[n-k-1]_q![n+k-1]_q!} \\
&=\frac{[n-1]_q!^2}{[n-k]_q![n+k]_q!}\left((1-q^n)^2-(1-q^{n-k})(1-q^{n+k})\right) \\
&=\frac{[n-1]_q!^2(1-q^k)^2q^{n-k}}{[n-k]_q![n+k]_q!}.
\end{align*}
Consequently, one obtains
\begin{align*} G_t(n)-G_t(n-1)
&=[n-1]_q!^2\sum_{k=1}^n(-1)^{k-1}\frac{(1+q^k)q^{\binom{k}2+tk+n-k}}{[k]_q^{2t-2}[n-k]_q![n+k]_q!} \\
&=\frac{[n-1]_q!^2}{[2n]_q!}\sum_{k=1}^n(-1)^{k-1}\frac{(1+q^k)q^{\binom{k}2+tk+n-k}\qbin{2n}{n-k}}{[k]_q^{2(t-1)}} \\
&=\frac{q^n}{[n]_q^2}\left(\frac{[n]_q!^2}{[2n]_q!}\sum_{k=1}^n(-1)^{k-1}\frac{(1+q^k)q^{\binom{k}2+(t-1)k}\qbin{2n}{n-k}}{[k]_q^{2(t-1)}}\right).
\end{align*}
That means, $G_t(n)-G_t(n-1)=\frac{q^n}{[n]_q^2}G_{t-1}(n)$.

\smallskip
\noindent 
Finally, we look into $H_t(n)$. We first isolate the terms with $k_{2t}=n$, and then the terms with $k_{2t-1}=n$, in such a way that
\begin{align}\label{e1} H_t(n)
&=\frac{q^n}{[n]_q}\sum_{1\leq k_1\leq\cdots\leq k_{2t-1}\leq n}\frac{q^{k_1+\cdots+k_{2t-1}}+q^{k_2+\cdots+k_{2t-2}}}{[n+k_1]_q[k_2]_q\cdots[k_{2t-1}]_q} \nonumber\\
& \qquad +\sum_{1\leq k_2\leq\cdots\leq k_{2t}\leq n-1}\,\,\,\sum_{k_1=1}^{k_2}\frac{q^{n+k_1+\cdots+k_{2t-1}}+q^{k_2+\cdots+k_{2t}}}
{[n+k_1]_q[k_2]_q\cdots[k_{2t}]_q}\nonumber\\
&=\frac{q^n}{[n]_q^2}H_{t-1}(n)
+\frac{q^n}{[n]_q}\sum_{1\leq k_1\leq\cdots\leq k_{2t-1}\leq n-1}\frac{q^{k_1+\cdots+k_{2t-1}}+q^{k_2+\cdots+k_{2t-2}}}
{[n+k_1]_q[k_2]_q\cdots[k_{2t-1}]_q}\nonumber\\
& \qquad +\sum_{1\leq k_2\leq\cdots\leq k_{2t}\leq n-1}\,\,\,\sum_{k_1=1}^{k_2}\frac{q^{n+k_1+\cdots+k_{2t-1}}+q^{k_2+\cdots+k_{2t}}}
{[n+k_1]_q[k_2]_q\cdots[k_{2t}]_q}.
\end{align}
On the other hand, rescaling  $k_1'=k_1-1$ in $H_t(n-1)$ we get
\begin{equation}\label{e2} 
H_t(n-1)=\sum_{1\leq k_2\leq\cdots\leq k_{2t}\leq n-1}\,\,\,\sum_{k_1'=0}^{k_2-1}\frac{q^{n+k_1'+\cdots+k_{2t-1}}+q^{k_2+\cdots+k_{2t}}}
{[n+k_1']_q[k_2]_q\cdots[k_{2t}]_q}.
\end{equation}
Comparing the inner sums in \eqref{e1} and  \eqref{e2} we find
\begin{align*}\sum_{k_1=1}^{k_2}&\frac{q^{n+k_1+\cdots+k_{2t-1}}+q^{k_2+\cdots+k_{2t}}}
{[n+k_1]_q[k_2]_q\cdots[k_{2t}]_q}-\sum_{k_1'=0}^{k_2-1}\frac{q^{n+k_1'+\cdots+k_{2t-1}}+q^{k_2+\cdots+k_{2t}}}
{[n+k_1']_q[k_2]_q\cdots[k_{2t}]_q}\\
&=q^{k_2}\cdot \frac{q^{n+k_3+\cdots+k_{2t-1}}+q^{k_4+\cdots+k_{2t}}}{[n+k_2]_q[k_2]_q\cdots[k_{2t}]_q}
-\frac{1}{[n]_q}\cdot \frac{q^{n+k_3+\cdots+k_{2t-1}}+q^{k_2+\cdots+k_{2t}}}{[k_2]_q\cdots[k_{2t}]_q}.
\end{align*}
Hence
\begin{align*} &H_t(n)-H_t(n-1)
=\frac{q^n}{[n]_q^2}H_{t-1}(n)+\frac{q^n}{[n]_q}\sum_{1\leq k_1\leq\cdots\leq k_{2t-1}\leq n-1}\frac{q^{k_1+\cdots+k_{2t-1}}+q^{k_2+\cdots+k_{2t-2}}}{[n+k_1]_q[k_2]_q\cdots[k_{2t-1}]_q} \\
&+\sum_{1\leq k_2\leq\cdots\leq k_{2t}\leq n-1}
\Big(q^{k_2}\cdot \frac{q^{n+k_3+\cdots+k_{2t-1}}+q^{k_4+\cdots+k_{2t}}}{[n+k_2]_q[k_2]_q\cdots[k_{2t}]_q}
-\frac{1}{[n]_q}\cdot \frac{q^{n+k_3+\cdots+k_{2t-1}}+q^{k_2+\cdots+k_{2t}}}{[k_2]_q\cdots[k_{2t}]_q}\Big). 
\end{align*}
So far, what we gathered takes the form $H_t(n)-H_t(n-1)=\frac{q^n}{[n]_q^2}H_{t-1}(n)+\mathbf{W}$
where $\mathbf{W}$ comprises the last two multi-sums (we renamed $k_2,\dots,k_{2r}\mapsto k_1,\dots,k_{2r-1}$ in the latter): 
\begin{align*}
\mathbf{W}:=
&\sum_{1\leq k_1\leq\cdots\leq k_{2t-1}\leq n-1}\Big(
\frac{q^{n}}{[n]_q}\cdot\frac{q^{k_1+\cdots+k_{2t-1}}+q^{k_2+\cdots+k_{2t-2}}}{[n+k_1]_q[k_2]_q\cdots[k_{2t-1}]_q} \\
&+
q^{k_1}\cdot \frac{q^{n+k_2+\cdots+k_{2t-2}}+q^{k_3+\cdots+k_{2t-1}}}{[n+k_1]_q[k_1]_q\cdots[k_{2t-1}]_q}
-\frac{1}{[n]_q}\cdot\frac{q^{n+k_2+\cdots+k_{2t-2}}+q^{k_1+\cdots+k_{2t-1}}}{[k_1]_q\cdots[k_{2t-1}]_q}
\Big).
\end{align*}
We now rearrange the expression $\mathbf{W}$ as follows
\begin{align*} 
\mathbf{W}
&=\sum_{1\leq k_1\leq\cdots\leq k_{2r-1}\leq n-1}
\left(\frac1{[n]_q[n+k_1]_q}+\frac{q^{k_1}}{[n+k_1]_q[k_1]_q}-\frac1{[n]_q[k_1]_q}\right)
\frac{q^{n+k_2+\cdots+k_{2t-2}}}{[k_2]_q\cdots[k_{2t-1}]_q} \\
&+\sum_{1\leq k_1\leq\cdots\leq k_{2t-1}\leq n-1}
\left(\frac{q^n}{[n]_q[n+k_1]_q}+\frac1{[n+k_1]_q[k_1]_q}-\frac1{[n]_q[k_1]_q}\right)
\frac{q^{k_1+k_3+\cdots+k_{2t-1}}}{[k_2]_q\cdots[k_{2t-1}]_q} 
\end{align*}
and realize that both terms in parentheses vanish, leading to the conclusion $\mathbf{W}=0$:
$$\frac1{[n]_q[n+k_1]_q}+\frac{q^{k_1}}{[n+k_1]_q[k_1]_q}-\frac1{[n]_q[k_1]_q}
=\frac{1-q^{k_1}+q^{k_1}(1-q^n)-1+q^{n+k_1}}{[n]_q[n+k_1]_q[k_1]_q}=0,$$
and the other one is obtained by swapping the indices $n$ and $k_1$.

\noindent 
Thus,  
$$H_t(n)-H_t(n-1)=\frac{q^n}{[n]_q^2}H_{t-1}(n).$$
The agreement in the initial conditions 
$$F_t(0)=G_t(0)=H_t(0)=0\quad\text{and}\quad
F_t(1)=G_t(1)=H_t(1)=\frac{q^t}{(1-q)^{2t}},$$ 
while satisfying the exact same recurrence relation, ensures the validity of the theorem. Therefore, the proof is complete.
\end{proof}

\begin{corollary} \label{tranMac} For an integer $t\geq1$, it holds that
\begin{align*}
\sum_{\mathcal{A}_t}\frac{q^{k_1+k_2+\cdots+k_t}}{\prod_{j=1}^r(1-q^{k_j})^2}
=\sum_{k=1}^{\infty}\frac{(-1)^{k-1}(1+q^k)q^{\binom{k}2+tk}}{(1-q^k)^{2t}}
=\sum_{\mathcal{A}_{2t-1}}\frac{k_1\cdot q^{k_1+k_3+\cdots+k_{2t-1}}}{\prod_{j=1}^{2t-1}(1-q^{k_j})}.
\end{align*}
\end{corollary}
\begin{proof} In Theorem \ref{T3}, divide through by $(1-q)^{2t}$ then increase $n\rightarrow\infty$. \end{proof}

\noindent
\begin{remark} Take the case $t=1$ in Theorem \ref{T3}, then $(1-q)^2F_1(n)=\sum_{k=1}^n\frac{q^k}{(1-q^k)^2}$ can easily be seen as a generating function: the coefficient of $q^m$ in this series is the sum of the {\it complimentary divisors} of the divisors of $m$ that are no larger than $n$. Simply notice
$$\sum_{k=1}^n\frac{q^k}{(1-q^k)^2}=\sum_{k=1}^n\sum_{j=1}^{\infty}jq^{kj}=\sum_{m=1}^{\infty}q^m\sum_{\substack{k\vert m \\ 1\leq k\leq n}}\frac{m}k.$$
\end{remark}

\noindent
Next, we furnish yet another proof for part of Corollary \ref{tranMac} based on an identity of Jacobi.

\begin{proposition} For $t\geq1$, we have the identity
$$\sum_{1\leq m_1 \leq\dots\leq m_t<\infty} \prod_{j=1}^t\frac{q^{m_t}}{(1-q^{m_t})^2}
=\sum_{m\geq1}\frac{(-1)^{m-1}q^{\binom{m}2+tm}(1+q^m)}{(1-q^m)^{2t}}.$$
\end{proposition}
\begin{proof}
\noindent
Consider the expression 
$$\frac{(1-q^m)^2}{1-2q^m\cos(2x)+q^{2m}}=\frac1{1+4\frac{q^m}{(1-q^m)^2}\sin^2x}$$
and recall Jacobi's identity \cite[page 187]{Jacobi} 

\begin{align*}
\prod_{m\geq1}\frac{(1-q^m)^2}{1-2q^m\cos(2x)+q^{2m}}
=\sum_{m\geq1}(-1)^{m-1}\left\{\frac{q^{\binom{m}2}(1-q^m)(1-q^{2m})}{1-2q^m\cos(2x)+q^{2m}}\right\}.
\end{align*}
As a consequence, we may proceed to compute
\begin{align*}
\prod_{m\geq1}\frac1{1+4\frac{q^m}{(1-q^m)^2}\sin^2x}
&=\sum_{m\geq1}\frac{(-1)^{m-1}q^{\binom{m}2}(1-q^m)(1-q^{2m})}{1-2q^m\cos(2x)+q^{2m}} \\
&=\sum_{m\geq1}\frac{(-1)^{m-1}q^{\binom{m}2}(1-q^m)(1-q^{2m})}{(1-q^m)^2+4q^m\sin^2x} \\
&=\sum_{m\geq1}\frac{(-1)^{m-1}q^{\binom{m}2}(1+q^m)}{1+4\frac{q^m}{(1-q^m)^2}\sin^2x} \\
&=\sum_{m\geq1}(-1)^{m-1}q^{\binom{m}2}(1+q^m)\sum_{n\geq0}\frac{(-4)^nq^{nm}\sin^{2n}x}{(1-q^{2m})^{2n}} \\
&=\sum_{n\geq0}(-4)^n\sin^{2n}x\sum_{m\geq1}\frac{(-1)^{m-1}q^{\binom{m}2+nm}(1+q^m)}{(1-q^m)^{2n}}.
\end{align*}
Compare the coefficients in the last double sum against the derivation of equation \eqref{homog} (cf. Section 6 below) in arriving at the conclusion alluded to by the proposition.
\end{proof}

\begin{example} As a further utility of the above identity of Jacobi towards our multiple sums, we list a few formulas. Letting $x=\frac{\pi}2, \frac{\pi}4$ and $\frac{\pi}6$ one obtains, respectively,
\begin{align*}
\prod_{m\geq1}\frac{(1-q^m)^2}{(1+q^m)^2}&=\sum_{m\geq1}(-1)^{m-1}\frac{q^{\binom{m}2}(1-q^m)(1-q^{2m})}{(1+q^m)^2}
=\sum_{n\geq0}\sum_{m_1{\leq}\dots{\leq}m_n<\infty} \prod_{j=1}^n\frac{(-4)q^{m_j}}{(1-q^{m_j})^2}, \\
\prod_{m\geq1}\frac{(1-q^m)^2}{1+q^{2m}}&=\sum_{m\geq1}(-1)^{m-1}\frac{q^{\binom{m}2}(1-q^m)(1-q^{2m})}{1+q^{2m}}
=\sum_{n\geq0}\sum_{ m_1{\leq}\dots{\leq}m_n<\infty} \prod_{j=1}^n\frac{(-2)q^{m_j}}{(1-q^{m_j})^2}, \\
\prod_{m\geq1}\frac{(1-q^m)^2}{1-q^m+q^{2m}}&=\sum_{m\geq1}(-1)^{m-1}\frac{q^{\binom{m}2}(1-q^m)(1-q^{2m})}{1-q^m+q^{2m}}
=\sum_{n\geq0}\sum_{m_1{\leq}\dots{\leq}m_n<\infty} \prod_{j=1}^n\frac{(-1)q^{m_j}}{(1-q^{m_j})^2}.
\end{align*}
Observe that the first of these products generates the number of lattice points stationed on a circle of radius $\sqrt{n}$, for a positive integer $n$.
\end{example}

\section{Another $q$-variant of \eqref{mhsid}}

\noindent
This section offers a slightly different $q$-analogue of \eqref{mhsid} based on an iterated $q$-inverse correspondence between two sequences that are related as such.

\begin{lemma} \label{L1} If a sequence $b_k$ is a $q$-binomial transform of $a_k$, that is,
$$\sum_{k=1}^n (-1)^{k-1}\qbin{n}{k}a_k=b_n$$
then it is true that
$$\sum_{k=1}^n (-1)^{k-1}\qbin{n}{k}\frac{a_k\,q^{tk}}{[z+k]_q^{t}}
=\frac{1}{\qbin{z+n}{n}}\sum_{1\leq k_1\leq \cdots \leq k_{t}\leq n}
\frac{b_{k_1}\qbin{z+k_1}{k_1} q^{k_1+k_2+\dots+k_{t}}}{\prod_{j=1}^t[z+k_j]_q }.$$
\end{lemma}
\begin{proof} We begin by proving the identity for the case $r=1$, i.e.
\begin{equation}\label{caseone}
\sum_{k=1}^n (-1)^{k-1}\qbin{n}{k}\frac{a_kq^k}{[z+k]_q}=\frac{1}{\qbin{z+n}{n}}\sum_{k=1}^n\qbin{z+k}{k}\frac{b_kq^k}{[z+k]_q}.
\end{equation}
Indeed,
\begin{align*}
\sum_{k=1}^n\qbin{z+k}{k}\frac{b_kq^k}{[z+k]_q}
&=\sum_{k=1}^n\qbin{z+k}{k}\frac{q^k}{[z+k]_q}\sum_{m=1}^k (-1)^{m-1}\qbin{k}{m}a_m\\
&=\sum_{m=1}^n (-1)^{m-1} a_m\sum_{k=m}^n
\qbin{z+k}{k}\qbin{k}{m}\frac{q^k}{[z+k]_q}\\
&=\qbin{z+n}{n}\sum_{m=1}^n (-1)^{m-1} \qbin{n}{m}\frac{a_m \,q^m}{[z+m]_q}
\end{align*}
where in the last step we employed the identity
$$\sum_{k=m}^n
\qbin{z+k}{k}\qbin{k}{m}\frac{q^k}{[z+k]_q}
=\qbin{z+n}{n}\qbin{n}{m}\frac{q^m}{[z+m]_q}$$
which, in turn, is justified by the WZ-method: let  
$$F(m,k):=\qbin{z+k}{k}\qbin{k}{m}\frac{q^k}{[z+k]_q}\quad\text{and}\quad G(m,k):=F(m,k)\cdot \frac{[k-m]_q q^{k-m}}{[z+m]_q},$$ 
then simply check that
$F(m,k)=G(m,k+1)-G(m,k)$. Upon simplification, we obtain
$$\sum_{k=m}^nF(m,k)=G(m,n+1)-G(m,m)=G(m,n+1)=\qbin{z+n}{n}\qbin{n}{m}\frac{q^m}{[z+m]_q}$$
as desired. By applying the above identity \eqref{caseone}, successively, $t$ times finishes the proof.
\end{proof}

\noindent
The next result is a consequence of Lemma \ref{L1} and reproves the identity in \eqref{ATidB} as the special case $z=0$.
\begin{corollary} \label{C1} For all positive integers $n$ and $t$, we have
$$
\sum_{k=1}^n (-1)^{k-1}\qbin{n}{k}\frac{q^{\binom{k}{2}+(x+t)k}}{[z+k]_q^{t}\qbin{x+k}{k}}
=\frac{1}{\qbin{z+n}{n}}\sum_{1\leq k_1\leq \cdots \leq k_{t}\leq n}\frac{[k_1]_q\,q^x\qbin{z+k_1}{k_1}q^{k_1+k_2+\dots+k_{r}}}{[x+k_1]_q\, \prod_{i=1}^t[z+k_i]_q}.
$$
\end{corollary}
\begin{proof} Consider the sequences 
$$a_k=\frac{q^{\binom{k}{2}+xk}}{\qbin{x+k}{k}}\quad\text{and}\quad b_k=\frac{[k]_q \,q^x}{[x+k]_q}$$
then check that these sequences fulfill the condition of Lemma \ref{L1}: %(see also \cite[(4)]{MSS} and \cite[(1.6)]{X})
$$\sum_{k=1}^n (-1)^{k-1}\qbin{n}{k}a_k=b_n.$$
This, however, is provable by the WZ-pair
$$F(n,k)=(-1)^{k-1}\qbin{n}{k}\frac{a_k}{b_n}\quad\text{and}\quad 
G(n,k)=-F(n,k)\frac{[x+k]_q[k-1]_q }{[x+n]_q[n+1-k]_q}q^{n+1-k}$$
exhibiting the property that $F(n+1,k)-F(n,k)=G(n,k+1)-G(n,k)$. To complete the proof apply Lemma \ref{L1} and simplify the terms.
\end{proof}

\noindent
Moreover, applying Lemma \ref{L1} we are able to recover identity (\ref{MSSid}) and hence equation (\ref{Did}) of Dilcher.

\begin{corollary} For positive integers $n$ and $t$, we have
\begin{equation}
\sum_{k=1}^n \frac{(-1)^{k-1}\qbin{n}{k}[k]_q\,q^{\binom{k}{2}+tk}}{[z+k]_q^{t+1}}
=\frac{1}{\qbin{z+n}{n}}\sum_{1\leq k_1\leq \cdots \leq k_{t}\leq n}\frac{q^{k_1+k_2+\dots+k_{t}}}{\prod_{j=1}^t[z+k_j]_q}.
\end{equation}
\end{corollary}
\begin{proof} Use $a_k=\frac{[k]_q\,q^{\binom{k}2}}{[z+k]_q}$ and $b_k=\frac1{\qbin{z+k}{k}}$. Then, the required hypothesis
$$\sum_{k=1}^n (-1)^{k-1}\qbin{n}{k}a_k=b_n$$
is justified through the WZ-pair
$$F(n,k)=(-1)^{k-1}\qbin{n}{k}\frac{a_k}{b_n}\quad\text{and}\quad 
G(n,k)=-F(n,k)\frac{[z+k]_q[k-1]_q }{[n]_q[n+1-k]_q}q^{n+1-k}.$$
The assertion follows after a direct application of Lemma \ref{L1}.  
\end{proof}

\section{Symmetric functions and quasi-modularity}

\noindent
The sum of divisors of a positive integer $n$ is defined as $\sigma_1(n)=\sum_{d\,\vert\, n}d$ and it can be generated by
$$U_1(q)=\sum_{n\geq1}\sigma_1(n)\,q^n=\sum_{k\geq1}\frac{q^k}{(1-q^k)^2}.$$ 
MacMahon \cite[pages 75, 77]{MacMahon} generalized this notion by introducing 
\begin{align*}
U_t(q):&=\sum_{1\leq k_1<\cdots<k_t<\infty}\frac{q^{k_1+\cdots+k_t}}{(1-q^{k_1})^2\cdots(1-q^{k_t})^2} \end{align*}
and interpreting the series as follows: define the sum $b_{n,t}=\sum s_1\cdots s_t$ where the sum is taken over all possible ways of writing $n=s_1k_1+\cdots+s_tk_t$ while 
$1\leq k_1<k_2<\cdots<k_t$. This way, $U_t(q)=\sum_{n\geq1}b_{n,t}\,q^n$ and in particular $b_{n,1}=\sigma_1(n)$.

\smallskip
\noindent
In \cite{MacMahon}, MacMahon developed several properties and concepts that are interwoven with each other. In the same spirit, we like to investigate the infinite series from Corollary \ref{tranMac}:
\begin{align*}
V_t(q):&=\sum_{1\leq k_1\pmb{\leq}\cdots\pmb{\leq} k_t<\infty}\frac{q^{k_1+\cdots+k_t}}{(1-q^{k_1})^2\cdots(1-q^{k_t})^2}.
\end{align*}
We remind the reader that $V_t(q)$ results from letting $n\rightarrow\infty$ in the function $F_t(n)$ of Theorem \ref{T3}. Next, consider the modest case $t=2$ to record an immediate observation.

\begin{proposition} \label{V2U2} We have a generating function
$$V_2(q)-U_2(q)=\frac16\sum_{j\geq1}\left(\sum_{d\,\vert\, j}d^3-\sum_{d\,\vert\, j}d\right)\,q^j$$
for an excess in the sum of divisors.
\end{proposition}
\begin{proof}  Since $V_2(q)-U_2(q)=\sum_{k\geq1}\frac{q^{2k}}{(1-q^k)^4}$,  it suffices to expand the right-hand side using the geometric series to the effect that
\begin{align*}
\sum_{k\geq1}\frac{q^{2k}}{(1-q^k)^4}&=\sum_{k\geq1}q^{2k}\sum_{n\geq0}\binom{n+3}3\,q^{kn}=\sum_{\substack{k\geq1 \\ m\geq2}}q^{km}\binom{m+1}3
=\sum_{\substack{k\geq1 \\ m\geq1}}q^{km}\binom{m+1}3 \\
&=\sum_{N\geq1}q^N\sum_{d\,\vert\, N}\binom{d+1}3=\frac16\sum_{N\geq1}q^N\sum_{d\,\vert\, N}(d^3-d).
\end{align*}
The proof is complete. \end{proof}

\noindent
MacMahon \cite{MacMahon} also worked out the expansion of the product $\prod_{m=1}^{\infty}(1+4\frac{q^m}{(1-q^m)^2}\sin^2(x))$. For notational simplicity, consider the following product which we expand to obtain
\begin{align*}
\prod_{m=1}^{\infty}(1+a_mX)&=\sum_{t\geq0}\left(\sum_{1\leq m_1<m_2<\dots<m_t<\infty}a_{m_1}a_{m_2}\cdots a_{m_t}\right)X^t \\
&=\sum_{t\geq0}e_t(a_1,a_2,a_3,\dots)\,X^t
\end{align*}
where $e_t(a_1,a_2,\dots)$ is the {\em $t^{th}$-elementary symmetric function} in infinitely many variables. Therefore, if we replace $a_m=\frac{q^m}{(1-q^m)^2}$ and $X=4\sin^2(x)$ then
$$e_t(a_1,a_2,a_3,\dots)=\sum_{1\leq m_1<m_2<\dots<m_t<\infty}\prod_{j=1}^t\frac{q^{m_j}}{(1-q^{m_j})^2}=U_t(q)$$
and also that $X^t=4^t\sin^{2t}x$.

\smallskip
\noindent
In the same vain, expand the product
\begin{align*}
\prod_{m=1}^{\infty}\frac1{1-a_mX}&=\sum_{t\geq0}\left(\sum_{1\leq m_1{\leq} m_2{\leq}\dots{\leq}m_t<\infty}
a_{m_1}a_{m_2}\cdots a_{m_t}\right)X^t \\
&=\sum_{t\geq0} h_t(a_1,a_2,a_3,\dots)\, X^t
\end{align*}
so that $h_t(a_1,a_2,\dots)$ is the {\em $t^{th}$-complete homogeneous symmetric function} in infinitely many variables. Once more, if we choose $a_m=\frac{q^m}{(1-q^m)^2}$ and $X=-4\sin^2x$ then
\begin{equation} \label{homog}
h_t(a_1,a_2,a_3,\dots)=\sum_{1\leq m_1{\leq} m_2{\leq}\dots{\leq}m_t<\infty} \prod_{j=1}^t\frac{q^{m_j}}{(1-q^{m_j})^2}=V_t(q).
\end{equation}

\smallskip
\noindent
\begin{remark} With this set-up, it is possible to relate MacMahon's $U_t(q)$ with our $V_t(q)$, thanks to the below result well-known in the theory of symmetric functions.
\end{remark}

\begin{lemma} \label{ele-hom} For $t\geq1$, there is the identity
$$\sum_{i=0}^t(-1)^ie_i\,h_{t-i}=0.$$
\end{lemma}

\smallskip
\noindent
As an immediate consequence, we garner the following interesting property.

\begin{theorem} \label{quasi1} The functions $V_t(q)$ belong to the ring of quasi-modular forms of weight at most $2t$ for some congruence subgroup $\Gamma$ of $SL_2(\mathbb{Z})$.
\end{theorem}
\begin{proof} We already noted that $h_t=V_t(q)$ and $e_t=U_t(q)$, so the recurrence relation from Lemma \ref{ele-hom} extracts
$$V_t(q)=\sum_{i=1}^t(-1)^{i-1}U_i(q)V_{t-i}(q).$$
Then, invoke either \cite[Corollary 4]{Andrews-Rose} or \cite[Theorem 1.12]{Rose} to utilize quasi-modularity of the $U_i(q)$'s. The argument is now settled.
\end{proof}

\begin{remark} While the $V_t(q)$ are $q$-series in the ring generated by quasi-modular forms, we emphasize that none are genuine quasi-modular forms of weight $2t$. Instead, they all are linear combinations of quasi-modular forms with weights $\leq 2t$. Indeed, we have $E_2=1-24\sum_{n\geq 0}\sigma_1(n)q^n$ is a weight $2$ quasi-modular form \cite{Zagier} and $E_0=1$ is a weight $0$ modular form, and so we have that
$V_1(q)=\frac{E_0(q)-E_2(q)}{24}$. Namely, it is a sum of a weight $0$ modular form and a weight $2$ quasi-modular form.
\end{remark}

\begin{example} \label{useful_ex} As an illustration of Lemma \ref{ele-hom}, look at the next two recurrence relations for the functions $V_t(q)$:
\begin{align*}
V_2(q)&=\frac1{10}\left([7V_1(q)-1]V_1(q)+q\frac{d}{dq}V_1(q)\right), \\
V_3(q)&=\frac1{21}\left([19V_1(q)-3]V_2(q)-4V_1(q)^3+V_1(q)^2+q\frac{d}{dq}V_2(q)\right).
\end{align*}
Note that $V_1(q)$ \rm{(}\it and hence each $V_t(q)$\rm{)} \it is expressible as a combination of the quasi-modular form $E_2(q)=1-24\sum_{n\geq1}\sigma_1(n)q^n$, that is, $V_1(q)=\frac{1-E_2(q)}{24}$ \rm{(}see \cite{Zagier}\rm{).}
\end{example}

\section{Elliptic functions and umbral expansions}

\noindent
In expressing his arithmetical series by means of elliptic functions, MacMahon \cite[p. 76]{MacMahon} made an effective use of 
$$\mathbf{J}_t(q)=\sum_{m\geq0}(-1)^m(2m+1)^tq^{\binom{m+1}2}$$ 
while Ramanujan \cite[equation (9)]{Rama} toyed with the \emph{Lambert series} 
$$\mathbf{S}_t(q)=\sum_{m\geq1}\frac{m^tq^m}{1-q^m}$$ 
(we dropped the zeta function from Ramanujan's definition).
%Here are a few examples of possible expansions in terms of these elliptic functions
%\begin{align*}
%U_1(q)&=-\frac{\mathbf{J}_3-\mathbf{J}_1}{2^2\,3!\,\mathbf{J}_1}, \qquad V_1(q)=\mathbf{S}_1, \qquad
%U_2(q)=\frac{\mathbf{J}_5-(1^2+3^2)\mathbf{J}_3+1^23^2\mathbf{J}_1}{2^4\,5!\,\mathbf{J}_1}, \\
%V_2(q)&=\frac{\mathbf{J}_5-(1^2+3^2)\mathbf{J}_3+1^23^2\mathbf{J}_1}{2^4\,5!\,\mathbf{J}_1}+\frac{\mathbf{S}_3-\mathbf{S}_1}{3!}.
%\end{align*}

%Notice, too, that $U_1(q)=V_1(q)$.

\smallskip
\noindent
Introduce $\mathbf{G}_t(q)=\sum_{m\geq1}\frac{q^{tm}}{(1-q^m)^{2t}}$ and let $\sfk{n}{k}$ denote the (\emph{unsigned}) \emph{Stirling numbers of the first kind}. We are set to register the following property.

\begin{proposition} \label{Stir} It holds that
\begin{align} \label{Stir-eqn}
\mathbf{G}_t(q)&=\frac1{(2t-1)!}\sum_{k=0}^{t-1}(-1)^k\left(\sum_{j=-k}^k(-1)^j 
\sfk{t}{t-k+j} \cdot \sfk{t}{t-k-j}\right)\mathbf{S}_{2t-1-2k}(q).
\end{align}
\end{proposition}
\begin{proof} We write
$$u(t,k)=\sum_{j=-k}^k(-1)^j \sfk{t}{t-k+j} \cdot \sfk{t}{t-k-j}$$
to observe that $u(t,k)$ are related to the so-called {\sl central factorial numbers} (see A008955) and it is known that (see \cite[Section 3]{Merca})
$$\sum_{k=0}^{t-1}(-1)^ku(t,k)x^{2t-1-2k}=(x+t-1)_{2t-1}$$
where $(x)_n=x(x-1)\cdots(x-n+1)$. Hence the right-hand side of \eqref{Stir-eqn} can be written as
\begin{align*}
\frac1{(2t-1)!}\sum_{k=0}^{t-1}(-1)^ku(t,k)\mathbf{S}_{2t-1-2k}(q) 
&=\frac1{(2t-1)!}\sum_{k=0}^{t-1}(-1)^ku(t,k)\sum_{j\geq1}\frac{j^{2t-1-2k}q^j}{1-q^j}\\
&=\sum_{j\geq1}\frac{q^j}{1-q^j}\frac1{(2t-1)!}\sum_{k=0}^{t-1}(-1)^ku(t,k)j^{2t-1-2k}\\
&=\sum_{j\geq1}\binom{j+t-1}{2t-1}\frac{q^j}{1-q^j}\\
&=\sum_{j\geq1}\sum_{k\geq1}\binom{j+t-1}{2t-1}q^{jk}.
\end{align*}
On the other hand, the left-hand side of \eqref{Stir-eqn} is
\begin{align*}
\mathbf{G}_t(q)&=\sum_{k\geq1}\frac{q^{tk}}{(1-q^k)^{2t}}
=\sum_{k\geq1}\sum_{j\geq1}\binom{j+t-1}{2t-1}q^{jk}
\end{align*}
and this completes the proof.
\end{proof}

\noindent
In the present context, one remarkable result is MacMahon's \cite[equation (1)]{MacMahon} formula
$$2^{2t}(2t+1)!\,U_t(q)=(-1)^t\frac1{\mathbf{J}_1}\,\mathbf{J}(\mathbf{J}^2-1^2)(\mathbf{J}^2-3^2)\cdots(\mathbf{J}^2-(2t-1)^2)$$
where we interpret $\mathbf{J}^t(q)$ (umbrally) as the $q$-series $\mathbf{J}_t(q)=\sum_{m\geq0}(-1)^m(2m+1)^t\,q^{\binom{m+1}2}$. Below, we list a few of our own findings.

\begin{example}
There is a more succinct presentation of Proposition \ref{Stir} in a manner
\begin{align*}
(2t-1)!\,\mathbf{G}_t(q)
%&=\frac1{(2a-1)!}\sum_{k=0}^{a-1}(-1)^k\,e_k(1^2,2^2,\dots,(n-1)^2)\,\pmb{S}_{2a-1-2k}(q) \\
&=\mathbf{S}(\mathbf{S}^2-1^2)(\mathbf{S}^2-2^2)(\mathbf{S}^2-3^2)\cdots(\mathbf{S}^2-(t-1)^2),
\end{align*}
associating $\mathbf{S}^t(q)$ to the \emph{umbral notation} for $\mathbf{S}_t(q)$, and multiplication is being executed accordingly. It is also possible to reverse the expansion, i.e., we can depict the $\mathbf{S}_t(q)$'s as a combination of a finite number of $\mathbf{G}_t(q)$'s:
\begin{align*}
\sum_{k=1}^tT(t,k)\,(2k-1)!\,\mathbf{G}_k(q)=\mathbf{S}_{2t-1}(q)
\end{align*}
with $T(t,k)=2\sum_{i=1}^k\frac{(-1)^{k-i}i^{2t}}{(k-i)!(k+i)!}$ being the central factorial numbers \cite[Exercise 5.8]{Stan}. This claim, however, is a simple consequence of the property that
$$x^t=\sum_{k=1}^t T(t,k)\, x(x-1^2)(x-2^2)\cdots(x-(k-1)^2)$$
as inherited from the generating function
$$\sum_{t\geq0}T(t,k)x^t=\frac{x^k}{(1-1^2x)(1-2^2x)\cdots(1-k^2x)}.$$
\end{example}

\begin{example}
Incidentally, there is a very similar result found in \cite[Lemma 2.5]{Andrews-et-al} which we are able to improve its formulation from the implicit $\sum_{m\geq1}\frac{q^{tm}}{(1-q^m)^t}=\sum_{j=0}^{t-1}c_{t,j}\mathbf{S}_j(q)$ to the more explicit and compact state-of-affairs:
$$(t-1)!\,\sum_{m\geq1}\frac{q^{tm}}{(1-q^m)^t}=\mathbf{S}_0(\mathbf{S}-1)(\mathbf{S}-2)(\mathbf{S}-3)\cdots(\mathbf{S}-(t-1)).$$
\end{example}

\begin{example}
Furthermore, the authors in \cite[Theorem 2.1]{Andrews-et-al} show that there exists  some polynomial $M_t\in\mathbb{Q}[x_1,\dots,x_t]$ such that 
$$\sum_{m\geq1}\frac{(-1)^{m-1}q^{\binom{m+1}2}}{(1-q^m)^t\prod_{j=1}^m(1-q^m)}=M_t(\mathbf{S}_0(q),\mathbf{S}_1(q),\dots,\mathbf{S}_{t-1}(q)).$$
On the other hand, \cite[Theorem 3]{D} proves 
\begin{align} \label{missed}
\sum_{m\geq1}\frac{(-1)^{m-1}q^{\binom{m+1}2}}{(1-q^m)^t\prod_{j=1}^m(1-q^m)}
=\sum_{i=1}^t\left\{\sum_{j=0}^{t-i}\binom{t-1}{j+i-1}\frac{\sfk{j+i}{i}}{(j+i)!}\right\}\mathbf{R}_i(q)
\end{align}
where $\mathbf{R}_t(q):=\sum_{m\geq1}m^tq^m\prod_{j\geq m+1}(1-q^j)$.

\smallskip
\noindent
We discover here, too, that there is a ``missed opportunity" in \eqref{missed}: one may write instead
\begin{align*} t! \,\sum_{m\geq1}\frac{(-1)^{m-1}q^{\binom{m+1}2}}{(1-q^m)^t\prod_{j=1}^m(1-q^m)}
&=\mathbf{R}(\mathbf{R}+1)(\mathbf{R}+2)\cdots(\mathbf{R}+t-1).
\end{align*}
\end{example}

\section{Congruences for generalized divisor sums}

\noindent
Let $M(t,n)$ denote the coefficients in the power series expansion of the three equivalent functions from Corollary \ref{tranMac}.  We chose the single sum (instead of the multiple sums)
\begin{align*}
\sum_{n\geq0}M(t,n)\,q^n:=\sum_{k\geq1}\frac{(-1)^{k-1}(1+q^k)\,q^{\binom{k}2+tk}}{(1-q^k)^{2t}}.
\end{align*}

\begin{theorem} \label{cong1} We have

\smallskip

\noindent
(i) If $t\equiv 0 \pmod{3}$ or $t\equiv 1 \pmod{3}$ then $3\,\vert \, M(t,3n+2)$.

\noindent
(ii) If $t\equiv 0 \pmod{5}$ then $5\,\vert \, M(t,5n+2)$ and $5\,\vert\,M(t,5n+4)$.

\noindent
(iii) If $t\equiv 2 \pmod{5}$ then $5\,\vert \, M(t,5n+1)$ and $5\,\vert\,M(t,5n+3)$.

\noindent
(iv) If $t\equiv 2 \pmod{7}$ then $7\,\vert \, M(t,7n+1)$.

\noindent
(v) If $t\equiv 3 \pmod{7}$ then $7\,\vert \, M(t,7n+1)$, $7\,\vert \, M(t,7n+2)$,  and $7\,\vert \, M(t,7n+6)$.

\end{theorem}
\begin{proof} Actually, we declare a stronger statement: all the claims hold true term-by-term in the generating function for $M(t,n)$. 
For each fixed $k\geq1$, we have
\begin{align*}\sum_{n\geq0}\Phi(t,k,n)\,q^n&:= \frac{(1+q^k)\,q^{\binom{k}2+tk}}{(1-q^k)^{2t}}=\sum_{m\geq0}\binom{m+2t-1}{2t-1}(1+q^k)q^{\binom{k}2+(m+t)k}\\
&=\sum_{m\geq0}\left(\binom{m+2t-1}{2t-1}+\binom{m-1+2t-1}{2t-1}\right)q^{\binom{k}2+(m+t)k}.
\end{align*}

\noindent (i) Denote $\delta(t):=\binom{m+2t-1}{2t-1}+\binom{m-1+2t-1}{2t-1}$. By \emph{Lucas' theorem},  if $t=3s$ then 
\begin{align*}
\delta(3s)&=\binom{m+3(2s-1)+2}{3(2s-1)+2}+\binom{m+3(2s-1)+1}{3(2s-1)+2} \equiv \binom{m+2}{2}+\binom{m+1}{2} \\
&=(m+1)^2 \not\equiv 0 \Leftrightarrow m\equiv 0,1 \pmod{3}.
\end{align*}
On the other hand, it is easy to verify that for all $k$,
$$\binom{k}2+(m+3s)k\equiv \text{$0$, $m$, $2m+1$} \pmod{3}.$$
Thus, for $m\equiv 0, 1\pmod 3$, we have $\binom{k}2+(m+3s)k\not \equiv 2 \pmod{3}$ as desired. The outcome is $3\,\vert\, \Phi(3s,k,3n+2)$ for any $k$. Finally, summing over all $k\geq1$, we are done with the first claim (i).

\smallskip

\noindent 
The remaining claims (congruences) are obtained \emph{mutatis mutandi}. 

\smallskip

\noindent (ii) If $t=5s$ then a similar analysis shows
\begin{align*}
\delta(5s)&=\binom{m+5(2s-1)+4}{5(2s-1)+4}+\binom{m+5(2s-1)+3}{5(2s-1)+4} \equiv \binom{m+4}{4}+\binom{m+3}{4} \\
&=\frac{(m+1)(m+2)^2(m+3)}{12} \equiv 3(m-2)(m-3)^2(m-4) \\
& \not\equiv 0 \Leftrightarrow m\equiv 0,1 \pmod{5}.
\end{align*}
On the other hand, it is easy to verify that for all $k$,
$$\binom{k}2+(m+5s)k\equiv \text{$0$, $m$, $2m+1$, $3m+3$, $4m+1$}\pmod{5}.$$
Thus, for $m\equiv 0, 1\pmod 5$, we have $\binom{k}2+(m+5s)k\not \equiv 2, 4\pmod{5}$ as desired.

\smallskip

\noindent (iii) If $t=5s+2$ then
\begin{align*}
\delta(5s+2)&=\binom{m+5(2s)+3}{5(2s)+3}+\binom{m+5(2s)+2}{5(2s)+3} \equiv \binom{m+3}{3}+\binom{m+2}{3} \\
&=\frac{(2m+3)(m+2)(m+1)}{6} \equiv -3(m-1)(m-3)(m-4)\\
&\not\equiv 0 \Leftrightarrow m\equiv 0,2 \pmod{5}.
\end{align*}
It is so \emph{facile} to verify that for all $k$,
$$\binom{k}2+(m+5s+2)k\equiv \text{$0$, $m+2$, $2m$, $3m-1$, or $4m-1$}\pmod{5}.$$
Therefore, for $m\equiv 0,2 \pmod 5$, we obtain $\binom{k}2+(m+5s+2)k\not \equiv 1,3\pmod{5}$.

\smallskip
\noindent (iv) If $t=7s+2$ then
\begin{align*}
\delta(7s+3) &\equiv -2(m-2)(m-5)(m-6) \not\equiv 0 \Leftrightarrow m\equiv 0,1,3,4 \pmod{7},
\end{align*}
and for $m\equiv 0,1,3,4 \pmod 7$ and any $k$, we conquer claim (iv) due to
$$\binom{k}2+(m+7s+2)k\not \equiv 1\pmod{7}.$$

\smallskip
\noindent (v) If $t=7s+3$ then
\begin{align*}
\delta(7s+3) &\equiv 2(m-1)(m-3)(m-4)(m-5)(m-6) \not\equiv 0 \Leftrightarrow m\equiv 0,2 \pmod{7},
\end{align*}
and for $m\equiv 0,2 \pmod 7$, there follows $\binom{k}2+(m+7s+3)k\not \equiv 1,2,6\pmod{7}$.
This concludes the proof of claim (v) and the theorem.
\end{proof}

\smallskip
\noindent
Recall the common notation $\sigma_s(n)=\sum_{d\,\vert\,n}d^s$ for the \emph{power-sum of divisors} of $n$.

\begin{lemma} \label{before_M84} We have the representations
\begin{align}
V_3(q)&=\frac{1}{1920}\sum_{n\geq 0}\left((40n^2+60n+9)\sigma_1(n)
-70(n+1)\sigma_3(n)+31\sigma_5(n)\right)q^n, \label{oneV3} \\
U_3(q)-V_3(q)&=\frac1{1920}\sum_{n\geq0}((-160n+28)\sigma_1(n)+(40n+120)\sigma_3(n)-28\sigma_5(n))q^n, \label{oneU3V3} \\
U_4(q)&=\frac{1}{967680}\sum_{n\geq 0}\big((-840n^3+5880n^2-9870n+3229)\sigma_1  \nonumber \\
&\quad \quad \quad +(756n^2-4410n+4935)\sigma_3+(-126n+441)\sigma_5+5\sigma_7\big)q^n. \label{oneU4}
\end{align}
\end{lemma}
\begin{proof} Let's recall the formulation of the power-sum divisor functions in terms of the three quasi-modular forms as:
$$\sum_{n\geq 0}\sigma_1(n)q^n=\frac{1-E_2(q)}{24}, 
\quad \sum_{n\geq 0}\sigma_3(n)q^n=\frac{E_4(q)-1}{240}, \quad
\sum_{n\geq 0}\sigma_5(n)q^n=\frac{1-E_6(q)}{504}.$$ 
If we denote the modular derivative by $D=q\frac{d}{dq}$, then Ramanujan's valuable formulas \cite[page 181, equation (30)]{Rama} can be brought to bear
\begin{align*}
DE_2=\frac{E_2^2-E_4}{12}, \qquad DE_4=\frac{E_2E_4-E_6}3, \qquad DE_6=\frac{E_2E_6-E_4^2}2.
\end{align*}
Being mindful of $V_1=\sum\sigma_1(n)q^n$, revive the formulas from Example \ref{useful_ex} of Section 6:
\begin{align} \label{oneV2}
V_2&=\frac1{10}(7V_1-1)V_1+\frac1{10}DV_1, \\ 
V_3&=\frac1{21}(19V_1-3)V_2-\frac4{21}V_1^3+\frac1{21}V_1^2+\frac1{21}DV_2. \label{twoV3}
\end{align}
Combining all of the above relations and after a direct (though routine) calculation, we are able to secure that both equation \eqref{oneV3} and equation \eqref{twoV3} are equal to
\begin{align*}
\frac{367}{967680}-\frac1{5120}E_2-\frac1{9216}E_2^2-\frac1{82944}E_2^3 -\frac1{23040}E_4-\frac1{69120}E_2E_4-\frac1{181440}E_6,
\end{align*}
where we used properties such as $\sum_{n\geq 0} n\sigma_s(n)q^n=D\sum_{n\geq 0}\sigma_s(n)q^n$.

\smallskip
\noindent
Recalling $e_t=U_t$ and $h_t=V_t$, by Lemma \ref{ele-hom} we have $U_3-V_3=V_1^3-2V_2V_1$. Now, use equation \eqref{oneV2} so that $U_3-V_3=-\frac25V_1^3+\frac15V_1^2-\frac15V_1\cdot DV_1$. It suffices to verify
$$\frac{-2V_1^3+V_1^2-V_1\cdot DV_1}5
=\frac{(-160D+28)\sum\sigma_1q^n+(40D+120)\sum\sigma_3q^n-28\sum\sigma_5q^n}{1920}.$$
This, however, pertains a similar procedure as in the first part above wherein direct computation shows both sides agree with
$$\frac{11}{34560} -\frac7{11520}E_2+ \frac1{3456}E_2^2 + \frac1{34560}E_2 E_4 -\frac1{34560}E_4.$$
To prove \eqref{oneU4}, we may use a recurrence from \cite[Corollary 3]{Andrews-Rose} ($A_k(q)$ renamed $U_k(q)$ here),
$$U_t(q)=\frac1{2t(2t+1)}\left[(6U_1(q)+t(t-1))U_{t-1}(q)-2DU_{t-1}(q)\right]$$
together with Ramanujan's \cite[Table IV]{Rama} formulas for $\sum\sigma_a(j)\sigma_b(n-j)$ and also the above relationships between the Eisenstein series $E_t(q)$ and sum of divisors $\sigma_s(q)$. We chose the quicker way: MacMahon \cite[page 104]{MacMahon} (typo corrected) has derived this already!
\end{proof}

\begin{theorem}  \label{M84} If $t=1, 2$ or $3$ then $7\,\vert \, M(t,8n+4)$.\end{theorem}

\begin{proof} The case $t=1$. By Corollary \ref{tranMac} (and the beginning of Section 6), 
$$\sum_{n\geq 0} M(1,n)q^n=\sum_{k=1}^{\infty}\frac{q^k}{(1-q^k)^2}=U_1(q)=\sum_{n=1}^{\infty}\sigma_1(n)q^n.$$
Then, due of the arithmetic property of $\sigma_s(n)$, we obtain 
$$M(1,8n+4)=\sigma_1(4(2n+1))=\sigma_1(4)\sigma_1(2n+1)=7\sigma_1(2n+1),$$ 
which is indeed divisible by $7$.

\smallskip
\noindent 
The case $t=2$. From Corollary \ref{tranMac} and Lemma \ref{ele-hom} (see also Example \ref{useful_ex} in Section 6),
\begin{align*}
\sum_{n\geq 0} M(2,n)q^n=V_2(q)&=\frac1{10}\left([7V_1(q)-1]V_1(q)+q\frac{d}{dq}V_1(q)\right) \\
&\equiv 5\sum_{n=1}^{\infty}(n-1)\sigma_1(n)q^n\pmod{7}.&
\end{align*}
Then $M(2,8n+4)$ is divisible by $7$ because
$\sigma_1(8n+4)=7\sigma_1(2n+1)$.

\smallskip
\noindent 
The case $t=3$.  Since $V_3(q)=\sum_{n\geq0}M(3,n)q^n$, by equation \eqref{oneV3} of Lemma \ref{before_M84}, we obtain
\begin{align*}\sum_{n\geq 0} M(3,n)q^n&=\frac{1}{1920}\sum_{n\geq 0}\left((40n^2+60n+9)\sigma_1(n)
-70(n+1)\sigma_3(n)+31\sigma_5(n)\right)q^n\\
&\equiv 4\sum_{n\geq 0}\left((5n^2+4n+2)\sigma_1(n)+3\sigma_5(n)\right)q^n\pmod{7},
\end{align*}
it follows that $M(3,8n+4)$ is divisible by $7$ because  $\sigma_1(8n+4)=7\sigma_1(2n+1)$ and by direct calculation
$\sigma_5(8n+4)=\sigma_5(4(2n+1))= \sigma_5(4)\sigma_5(2n+1)=7\cdot 151\,\sigma_5(2n+1)$.
\end{proof}

\smallskip
\noindent
In the Introduction section, the $MO(t,n)$ are defined as the coefficients in the MacMahon's version of the multiple sum $U_t(q)$. However, we choose to implement yet an equivalent generating function from \cite[Corollary 2]{Andrews-Rose}. That is to say,
\begin{align*}
\sum_{n\geq0}MO(t,n)\,q^n
=\frac{(-1)^t}{(q)_{\infty}^3}\sum_{k\geq t}(-1)^k\frac{2k+1}{2t+1}\binom{k+t}{k-t}q^{\binom{k+1}2}.
\end{align*}

\begin{theorem} \label{MO51} If $t=2$ then $5\,\vert \, MO(t,5n+1)$.
\end{theorem} 
\begin{proof} Following Lemma \ref{ele-hom}, we obtain $U_2(q)=V_1(q)^2-V_2(q)$. The numbers $MO(2,n)$ and $M(2,n)$ are the power series coefficients of $U_2(q)$ and $V_2(q)$, respectively. Since $5\,\vert\,M(2,5n+1)$, by Theorem \ref{cong1} (iii), it suffices to prove that $5\,\vert\,c(5n+1)$ where we refer to $\sum_{n\geq0}c(n)\,q^n=V_1(q)^2$. In fact, $c(n)=\sum_{j=1}^{n-1}\sigma_1(j)\sigma_1(n-j)$ is a convolution.

\smallskip
\noindent
Next, use the \emph{Eisenstein series} $E_2(\tau)=1-24V_1(q)$ and $E_4(\tau)=1+240\sum_{n\geq0}\sigma_3(n)q^n$ and one of Ramanujan's identities \cite[page 181, equation (30)]{Rama},
$$q\frac{dE_2}{dq}=\frac{E_2^2-E_4}{12}.$$ 
After some rearrangement and reading off the corresponding coefficients, this leads to \cite[page 186, Table IV, identity 1]{Rama},
\begin{align} \label{MO251} 12 \sum_{j=1}^{n-1} \sigma_1(j)\sigma_1(n-j) = 5 \sigma_3(n) + \sigma_1(n) - 6 n \sigma_1(n). \end{align}
Computing modulo $5$ implies that $2c(5n+1)\equiv [1-6(5n+1)]\sigma_1(5n+1)\equiv0 \pmod5$.
\end{proof}

\begin{corollary} For each integer $n\geq0$, we have $\sigma_3(5n+1)\equiv\sigma_1(5n+1) \pmod 5$.
\end{corollary}
\begin{proof} By Proposition \ref{V2U2}, there holds $V_2(q)-U_2(q)=\frac16\sum_{j\geq 0}\sigma_3(j)q^j-\frac16\sum_{j\geq 0}\sigma_1(j)q^j$. Theorem \ref{cong1} (iii) and Theorem \ref{MO51} confirm $M(2,5n+1)\equiv MO(2,5n+1)\equiv0\pmod5$. So, the coefficients of $q^{5n+1}$ in $V_2(q)-U_2(q)$ becomes divisible by $5$. The proof follows.
\end{proof}

\begin{lemma}\label{BeforeMO252} (a) Let $p$ be a prime and let $n\not\equiv 0 \pmod{p}$.
\noindent If $a,b,k,j\in\mathbb{Z}$ are such that $k+j\equiv 0 \pmod{p-1}$ and $a+bn^j\equiv 0 \pmod{p}$ then
$$a\sigma_k(n)+b\sigma_j(n)\equiv 0 \pmod{p}.$$
(b) If $p\neq 2$ and $n$ not a quadratic residue modulo $p$ then $\sigma_{\frac{p-1}{2}}(n)\equiv 0 \pmod{p}$.
\end{lemma}
\begin{proof} By the very definition of the power-sum divisors,
$$a\sigma_k(n)+b\sigma_j(n)=a\sum_{d\vert n}d^k+b\sum_{d\vert n}\left(\frac{n}d\right)^j
=\sum_{d\vert n}\frac{ad^{k+j}+bn^j}{d^j}.$$
The assumption $n\not\equiv 0\pmod p$ and Fermat's Little theorem, $d^{p-1}\equiv1 \pmod{p}$. Hence
$$a\sigma_k(n)+b\sigma_j(n)\equiv\sum_{d\vert n}d^k(a+bn^j)\equiv 0\pmod{p}.$$
If $n$ is not a quadratic residue modulo $p$ then (by Euler's criterion) $n^{\frac{p-1}{2}}\equiv -1\pmod{p}$, and after letting $k=j=\frac{p-1}{2}$ and $a=b=1$, we find
$$2\sigma_{\frac{p-1}{2}}(n)\equiv\sum_{d\vert n}d^k(1+n^{\frac{p-1}{2}})\equiv 0\pmod{p}$$
concludes the proof for part (b) of the assertion.
\end{proof}

%\smallskip
\begin{theorem} If $t=2$ then $5\,\vert \, MO(t,5n+2)$.
\end{theorem}
\begin{proof} From Lemma \ref{ele-hom}, Proposition \ref{V2U2}, and equation \eqref{MO251} we infer the system 
$$\begin{cases}
\!\!\!\!&V_2+U_2=V_1^2 \\ 
\!\!\!\!&V_2-U_2=\frac16\sum_{n\geq 0}(\sigma_3(n)-\sigma_1(n))q^n  \\
\!\!\!\!&\qquad \,V_1^2=\frac1{12}\sum_{n\geq 0}(5\sigma_3(n)-(1-6n)\sigma_1(n))q^n. 
\end{cases}$$
Solving for $U_2(q)=\sum M(2,n)q^n$ and simplifying the result expressions leads to 
\begin{align*} 2M(2,5n+2)&=\frac14\sigma_3(5n+2)+\frac14\sigma_1(5n+2)-\frac12(5n+2)\cdot\sigma_1(5n+2) \\
&\equiv -\sigma_3(5n+2)+3\sigma_1(5n+2) \equiv 3(\sigma_1(5n+2)-2\sigma_3(5n+2)) \pmod5.
\end{align*}
Choose $p=5, k=1, j=3, a=1, b=-2$ and $n\rightarrow 5n+2\not\equiv0\pmod 5$. Then, we observe that Lemma \ref{BeforeMO252}(a) forces the last expression to vanishes modulo $5$.
\end{proof}

\begin{theorem} If $t=3$ then $7\,\vert \, MO(t,7n+3)$ and $7\,\vert \, MO(t,7n+5)$.
\end{theorem}
\begin{proof} Relying on equations \eqref{oneV3} and \eqref{oneU3V3} of Lemma \ref{before_M84}, we gather that
$$U_3(q)=\frac{1}{1920}\sum_{n\geq 0}\left((40n^2-100n+37)\sigma_1(n)
-(30n-50)\sigma_3(n)+3\sigma_5(n)\right)q^n.$$
Therefore, since $U_3(q)=\sum_{n\geq 0} MO(3,n)q^n$, we are lead to
\begin{align*}
MO(3,7n+3)&\equiv 3\sigma_1(7n+3)+\sigma_3(7n+3)-2\sigma_5(7n+3)\pmod{7}, \\
MO(3,7n+5)&\equiv -\sigma_1(7n+5)-\sigma_3(7n+5)+5\sigma_5(7n+5)\pmod{7}.
\end{align*}
Choosing the values $p=7, k=1, j= 5, a=3, b=-2$ and $n\rightarrow 7n+3\not\equiv 0\pmod7$, apply Lemma \ref{BeforeMO252}(a) to obtain $3\sigma_1(7n+3)-2\sigma_5(7n+3)\equiv 0\pmod 7$. Again, invoke Lemma \ref{BeforeMO252}(a) with
 $p=7, k=1, j= 5, a=-1, b=5, n\rightarrow 7n+5\not\equiv 0\pmod7$ to infer another congruence $\sigma_1(7n+5)-(7n+5)\sigma_5(7n+5)\equiv 0\pmod 7$.

\smallskip
\noindent
Since neither of the numbers $7n+3$ and $7n+5$ is a quadratic residue modulo $7$, we may readily benefit from Lemma \ref{BeforeMO252}(b): 
$\sigma_3(7n+3)\equiv \sigma_3(7n+5)\equiv 0 \pmod{7}$. Indeed, we have witnessed enough reliable verity to reach the desired conclusion.
\end{proof}

\smallskip
\noindent
\begin{theorem} \label{MO4116} If $t=4$ then $11\,\vert \, MO(t,11n+6)$.
\end{theorem}
\begin{proof} For the present task, we revive a tool from Lemma \ref{before_M84} by way of identity \eqref{oneU4}:
\begin{align*}
U_4(q)&=\frac{1}{967680}\sum_{n\geq 0}\big((-840n^3+5880n^2-9870n+3229)\sigma_1(n) \\
&\quad \quad \quad +(756n^2-4410n+4935)\sigma_3(n)+(-126n+441)\sigma_5(n)+5\sigma_7(n)\big)q^n.
\end{align*}
Then, since $U_4(q)=\sum_{n\geq 0} MO(4,n)q^n$, we derive the congruence
\begin{align*}
MO(4,11n+6)&\equiv 7\sigma_3(11n+6)+7\sigma_5(11n+6))+6\sigma_7(11n+6)\pmod{11}.
\end{align*}
Because $11n+6$ is not a quadratic residue modulo $11$, once more Lemma \ref{BeforeMO252}(b) shows that $\sigma_5(11n+6)\equiv 0 \pmod{11}$. 

\smallskip
\noindent
In the next step, assume $p=11, k=3, j= 7, a=7, b=6, n\rightarrow 11n+6\not\equiv 0\pmod{11}$. Then, Lemma \ref{BeforeMO252}(a) can be invoked
to yield $7\sigma_3(11n+6)+6\sigma_7(11n+6)\equiv 0\pmod{11}$.
Combining the above congruences is sufficient to reach $MO(4,11n+6)\equiv0\pmod{11}$.
\end{proof}

\section{Conclusion}

\noindent
We believe that we have just scratched the surface of these MacMahon divisor sums and their extensions. In particular identity \eqref{partofcor3.1} is surely not the only identity of this nature leading to divisor sum identities. To be clear, the appropriate domain we are considering consists of multiple $q$-series wherein the numerator is of the form $q^{L(k_1,k_2,\dots,k_t)}$ with $L(k_1,k_2,\dots,k_t)$ a linear function of the $t$ indices of the $k$-fold series and each denominator is of the form
$$\prod_{j=1}^t(1-q^{a_jk_j})^{b_j}$$
with the $a_j$'s and $b_j$'s positive integers. 

\smallskip
\noindent
With each such series there is a ``conjugate" series equal to it. Our meaning of ``conjugate" is best illustrated by an example related to 
equation \eqref{partofcor3.1}:
\begin{align*}
\sum_{1\leq k_1<k_2<\cdots<k_{2t-1}}\frac{k_1\, q^{k_1+k_3+\cdots+k_{2t-1}}}{\prod_{j=1}^t(1-q^{k_j})}
&=\sum_{\substack{1\leq k_1\leq k_2\leq\cdots\leq k_{2t-1} \\ m_1, m_3, \dots,m_{2t-1}\geq1 \\ m_2, m_4, \dots, m_{2t-2}\geq0}}
k_1\,q^{m_1k_1+m_2k_2+\cdots+m_{2t-1}k_{2t-1}} \\
&=\sum_{M_1>M_2\geq M_3>M_4\geq\cdots\geq M_{2t-1}\geq1}
\frac{q^{M_1}}{(1-q^{M_1})\prod_{j=1}^{2t-1}(1-q^{M_j})}.
\end{align*}
The passage to the final expression is done by summing all the $k_j$ series, and then setting $M_i=m_i+m_{i+1}+\cdots+m_{2t-1}$.

\smallskip
\noindent
Clearly the above procedure can be applied to any series in this domain. The most important point to note is that the identity involving the two series in
equation \eqref{partofcor3.1} is not a ``conjugate" identity.

\smallskip
\noindent
Finally we should add a combinatorial comment. Let us return to the multi-sum expression \eqref{Mac1} in the case $t=2$. We are considering partitions into exactly two distinct parts each possibly appearing several times. For example,
$$5+5+3+3+3+3=2\cdot5+4\cdot3=2+2+2+2+2+3+3+3+3.$$
This seeming conjugation map is flawed because it is not an involution
$$4+4+2+2=2\cdot4+2\cdot2=2+2+2+2+2+2$$
while
$$5+5+1+1=2\cdot5+2\cdot1=2+2+2+2+2+2.$$
On the other hand, classic partition conjugation does preserve the feature of $t$-different parts. Thus the partition $5+5+3+3+3+3$ and its conjugate $6+6+6+2+2$ retain their respective Young diagrams
$$\young(~~~~~,~~~~~,~~~,~~~,~~~,~~~) \qquad \qquad \text{and}  \qquad \qquad  \young(~~~~~~,~~~~~~,~~~~~~,~~,~~)$$
The series ``conjugate" referred to in the example above actually corresponds to the faulty conjugate map and not the classic conjugate map.

\smallskip
\noindent
In summary, as we observed earlier, we only saw the tip of the iceberg. We look ahead to many more arithmetic and combinatorial discoveries.

\smallskip
\noindent
In the preceding section we succeeded in proving several congruences both for MacMahon's generalized divisor sums $U_t(q)$ as well as for our own $V_t(q)$. We wish to close the discussion by inviting the stimulated reader one tantalizing congruence which might benefit from a different (fresher) breed of techniques.

\begin{conjecture} If $t=10$ then $11\,\vert \, MO(t,11n+7)$.
\end{conjecture}

\bigskip
\noindent
{\bf Acknowledgments.} The first author appreciates Ken Ono and Olivia Beckwith for some useful discussions.

\end{document}